\documentclass[a4paper]{amsart}
\usepackage{amssymb}
\usepackage[alphabetic]{amsrefs}

\usepackage{yfonts, bbold}

\usepackage{a4wide, xcolor, enumerate,graphicx,todonotes,amsthm}

\usepackage[normalem]{ulem}

\usepackage{mathrsfs}
\usepackage{tabulary}
\usepackage{hyperref}

\hypersetup{
 colorlinks,
 linkcolor={red!50!black},
 citecolor={blue!50!black},
 urlcolor={blue!80!black}
}

\newcounter{abb}

\newcounter{tab}

\newcommand{\RR}{\mathbb{R}}

\newcommand{\KO}{\mathrm{KO}}

\newcommand{\CC}{\mathbb{C}}
\newcommand{\HH}{\mathbb{H}}
\newcommand{\UU}{\mathrm U}
\newcommand{\Sp}{\mathrm{Sp}}

\newcommand{\R}{\mathbb{R}}
\newcommand{\N}{\mathbb{N}}
\newcommand{\ZZ}{\mathbb{Z}}

\newcommand{\Inj}{\mathrm{Inj}}
\newcommand{\Iso}{\mathrm{Iso}}
\newcommand{\OO}{\mathrm{O}}
\newcommand{\CP}{\mathbb{C} P}

\newcommand{\HP}{\mathbb{H} P}

\newcommand{\NUMS}{\nu_S^M}
\newcommand{\skmu}{\mathbb{S}^{k-1}}

\newcommand{\Id}{\mathrm{Id}}

\DeclareMathOperator{\id}{\mathrm{id}}
\DeclareMathOperator{\scal}{\mathrm{scal}}
\DeclareMathOperator{\tr}{\mathrm{tr}}

\newtheorem{lem}{Lemma}[section]
\newtheorem{thm}[lem]{Theorem}
\newtheorem{prop}[lem]{Proposition}

\newtheorem{cor}[lem]{Corollary}

\newcommand{\pref}[1]{Proposition~\ref{#1}}
\newcommand{\lref}[1]{Lemma~\ref{#1}}
\newcommand{\dref}[1]{Definition~\ref{#1}}
\newcommand{\tref}[1]{Theorem~\ref{#1}}
\newcommand{\cref}[1]{Corollary~\ref{#1}}
\newcommand{\sref}[1]{Section~\ref{#1}}

\theoremstyle{definition}
\newtheorem{rem}[lem]{Remark}
\newtheorem{quest}[lem]{Question}

\newtheorem{dfn}[lem]{Definition}
\newtheorem{exa}[lem]{Example}

\newtheorem{set}[lem]{Setting}

\begin{document}

\title{Scalar positive immersions}
\author{Luis A. Florit}
\address{IMPA, Est. Dona Castorina, 110, Jardim Botânico 22460-320, Rio de Janeiro, RJ, BRAZIL}
\email{luis@impa.br}
\urladdr{}
\author{Bernhard Hanke}
\address{Universit\"at Augsburg, Institut f\"ur Mathematik, 86135 Augsburg, Germany}
\email{hanke@math.uni-augsburg.de}
\thanks{The first named author has been supported by CNPq - Brazil}
\thanks{The second named author has been supported by the Special Priority Program SPP 2026 ``Geometry at Infinity'' funded by the DFG}
\begin{abstract}
As shown by Gromov-Lawson and Stolz the only obstruction to the existence of positive scalar curvature metrics on closed simply connected manifolds in dimensions at least five appears on spin manifolds and is given by the non-vanishing of the $\alpha$-genus of Hitchin.

When unobstructed we shall realize  a positive scalar curvature metric by an immersion into Euclidean space whose dimension is uniformly close to the classical Whitney upper bound for smooth immersions.
Our main tool is an extrinsic counterpart of the well-known Gromov-Lawson surgery procedure for constructing positive scalar curvature metrics.
\end{abstract}



\maketitle

\section{Introduction}

One of the central results in positive scalar curvature geometry  \cites{gl,st} says that a closed simply connected manifold $M$ of dimension $n \geq 5$ admits a Riemannian metric of positive scalar curvature unless $M$ is spin and Hitchin's $\alpha$-genus $\alpha(M) \in \KO^{-n}$  is non-zero (see \cite{Hitchin}).
The purpose of our paper is to apply the ideas behind this result to  the classical problem of finding immersions into Euclidean space in low codimensions under certain curvature hypotheses.
We are interested here in positive scalar curvature.
In this work all manifolds and maps between manifolds are assumed to be smooth.

\begin{dfn}
We say that an immersion $f \colon M \to \R^N$ of some manifold $M$ is {\em scalar positive} if the Riemannian metric induced on $M$ by $f$ has positive scalar curvature.
\end{dfn}

The classical Nash isometric embedding theorem \cite{nash} implies that a closed (i.e., compact without boundary) Riemannian manifold $M$ of positive scalar curvature admits an  isometric, hence scalar positive, immersion into Euclidean space whose dimension depends quadratically on the dimension of $M$.
Our main result shows that, in the cases mentioned before, this dimension bound can be improved considerably if we do not restrict to a specific positive scalar curvature metric on $M$.

\begin{thm} \label{main}
Let $M$ be a closed simply connected manifold of dimension $n\geq 5$.
If $M$ is spin, assume further that $\alpha(M)=0$.
Then there exists a scalar positive immersion $M \to \R^{2n-1+\delta(n)}$ where
\[
 \delta(n) = \begin{cases}  \max\{0, 13 - \beta(n+6) \} \in \{ 0,\dots,12 \},  & \textrm{ if } M \text{ is spin}, \\
\max\{0,\, 9\, -\,\beta(n+4)\} \in \{ 0,\dots,8 \}, & \textrm{ if } M \text{ is not spin} .
 \end{cases}
\]
Here $\beta(m)$ denotes the number of digits $1$ in the dyadic expansion of $m \in \N$.
\end{thm}

Recall that $2n-1$ is Whitney's classical upper dimension bound for immersions of $n$-manifolds (for $n \geq 2$)  into Euclidean space.
The dimension bound for scalar positive immersions in \tref{main} increases the Whitney bound by at most twelve, and it is in fact equal to the Whitney bound in most dimensions.
However, the Whitney bound  is in general not sufficient for realizing scalar positive immersions.
Indeed, as we will see in \sref{nomsph}, the normal bundle of such an immersion $M \to \R^N$ splits off the line spanned by the nowhere vanishing mean curvature field.
Hence, by \cite{Hirsch}*{Theorem 6.4}, the manifold $M$ actually immerses into $\R^{N-1}$ if $\dim M < N-1$ where for non-compact connected $M$ the assumption $\dim M < N-1$ can be dropped by \cite{Hirsch2}*{Theorem 4.7.}.
This observation is illustrated by the following example.

\begin{exa}\label{cpn1} On the one hand, according to \tref{main}, for $m \geq 3$ the complex projective space $\CP^{m}$ admits a scalar positive immersion into $\R^{4m+11}$.
For $m \geq 5$ this improves the embedding dimension $m^2 + 2m$ of the isometric Veronese embedding of $\CP^m$ with the Fubini-Study metric, which is of positive scalar curvature.
On the other hand, for $\ell \geq 1$ and $m = 2^{\ell}$ the manifold $\CP^m$ does not immerse into $\R^{4m-2}$ by \cite{sand}*{Theorem 4}.
Hence it does not admit a scalar positive immersion into $\R^{4m-1}$, and the Whitney bound is not sufficient for realizing a scalar positive immersion.
\end{exa}

These considerations lead us to the following interesting open problem.

\begin{quest} \label{openquest} Let $M$ be a closed manifold admitting both a positive scalar curvature metric and an immersion $M \to \R^N$.
Does $M$ admit a scalar positive immersion $M \to \R^{N+1}$?
\end{quest}

\begin{rem} \label{remquest}
The corresponding question for non-compact connected manifolds has an affirmative answer due to Gromov's $h$-principle; see \pref{hprinc} below.
In combination with \cite{Hirsch2}*{Theorem 4.7.},  this implies that a   non-compact connected parallelizable manifold $M$ of dimension $n \geq 2$ admits a scalar positive  immersion into $\R^{n+1}$, but clearly not into $\R^n$.
\end{rem}

The main ingredient of our proof of \tref{main} is the following extrinsic version of the surgery result proved independently by Gromov-Lawson \cite{gl} and Schoen-Yau \cites{SY}.

\begin{thm}\label{extr_surg}
Let $f\colon M \to \R^N$ be a scalar positive immersion with $n = \dim M$.
Assume that $\hat M$ is obtained from $M$ by a surgery along an embedded sphere $S^d \subset M$ of codimension $n - d \geq 3$.
If $N\geq n+d+2$, then there exists a scalar positive immersion $\hat f \colon \hat M \to \R^N$.
Furthermore, the immersion $\hat f$ may be assumed to coincide with $f$ outside an arbitrarily small neighborhood of $S^d$ in $M$.
\end{thm}

Our paper is organized as follows.
In \sref{bundlescaling} we construct scalar positive immersions of total spaces of fibre bundles whose fibres are equipped with positive scalar curvature metrics.
This  uses a variation of the well-known fibrewise shrinking process in Riemannian submersions with scalar positive fibres.
Example \ref{Veronese} provides scalar positive immersions of total spaces of $\CP^2$-bundles and $\HP^2$-bundles  from which the scalar positive immersions  in Theorem \ref{main} will ultimately be constructed by extrinsic surgeries in codimensions at least $3$.
In Sections \ref{nomsph} and \ref{bending}, which form the technical core of our paper, we study the two types of local deformations near closed embedded submanifolds that are required for the extrinsic surgery  process in \tref{extr_surg}.
At first, we use the local flexibility lemma proved by B\"ar and the second named author \cite{bh} to bring a given scalar positive immersion into a particularly convenient form around  a  submanifold; see \pref{local_deform}.
Then, in \pref{bending_curve} we construct the appropriate bending profiles required for the extrinsic surgery.
After these preparations the proofs of  Theorems \ref{extr_surg} and \ref{main}  are  completed in  \sref{extrsurg}

\medskip

\textit{Acknowledgments:}
The second named author is grateful to IMPA, Rio de Janeiro, and to the Courant Center (NYU), New York, for their hospitality when parts of this research were carried out.
Also he wishes to thank Misha Gromov for a number of stimulating remarks.

\section{Scalar positive immersions via normal bundle scaling} \label{bundlescaling}

In this section we obtain scalar positive immersions of total spaces of fibre bundles from which  the manifolds $M$  in \tref{main} can be obtained by extrinsic surgeries in codimensions at least $3$.

\medskip

Let us begin by establishing the basic setup.

\begin{set} \label{assumption} Let $B$ be a compact $\ell$-dimensional manifold, possibly with boundary, and let $E \to B$ be a Euclidean vector bundle of rank $m$.
Furthermore, let $X \subset \R^m$ be a closed (i.e., compact embedded without boundary) submanifold and $V \to B$ be a sub-fibre bundle of $E \to B$ with fibre $X$ such that around each point in $B$ there exists an orthogonal vector bundle trivialization $\Psi \colon E|_U \stackrel{\cong}{\to} U \times \R^m$ satisfying
\begin{equation} \label{subbundle}
 \Psi( V \cap E|_U) = U \times X \subset U \times \R^m \, .
\end{equation}
\end{set}
\noindent In particular, the structure group of $V \to B$ reduces to the isometry group of the induced Riemannian metric on $X \subset \R^m$, which we denote by $h$.

\begin{dfn} \label{def:compatible} Let $g$ be a Riemannian metric on an open neighborhood of the zero section $B = B \times 0 \subset E$.
We say that $g$ is {\em compatible} with the Euclidean structure of $E \to B$ if there is an orthogonal decomposition
\begin{equation} \label{orthbas}
 g|_{B \times 0} = g_{B} \oplus \langle \,\cdot\,, \cdot \,\rangle_{E}
\end{equation}
where $g_B$ is the metric on $B$ induced by $g$, and $\langle\, \cdot\,, \cdot \,\rangle_{E}$ is the given bundle metric on $E$ considered as a subbundle of $ TE|_B$.
\end{dfn}

For $\lambda > 0$ we denote by $\lambda V \subset E$ the image of the fibrewise dilation of $V$ by the factor $\lambda$.

\begin{prop}\label{tub_imm} Let $g$ be a Riemannian metric on $E$ which is compatible with the Euclidean structure of $E$. If $\scal_h > 0$, then there exists $\lambda_0 >0 $ such that  for all $0 < \lambda \leq \lambda_0$ the induced metric on $\lambda V \subset (E,g)$ has positive scalar curvature.
\end{prop}

\begin{proof}
Without loss of generality we can assume that $X \subset \R^m$ is contained in the closed unit ball $D_1^m \subset \R^m$.
Consider a compact subset $K \subset B$ contained in an open subset $K \subset U \subset B$ which admits an orthogonal vector  bundle trivialization $\Psi \colon E|_U \stackrel{\cong}{\to} U \times \R^m$ satisfying \eqref{subbundle} and a local manifold chart $\phi \colon U \to \phi(U) \subset \R^\ell$.
Setting $n = \ell + m$, we obtain a manifold chart
\[
 \Phi \colon E|_U \stackrel{\Psi}{\cong} U \times \R^m \stackrel{\phi \times \id}{\approx} \phi(U) \times \R^m \subset \R^{n} \, .
\]
Fix standard coordinates $(x^1, \ldots, x^{\ell})$ and $(x^{\ell+1}, \ldots, x^n)$ on $\R^\ell$ and $\R^m$.
With respect to the local manifold chart $\Phi$ the metric $g$ has smooth components $g_{ij}=g\left(\frac{\partial}{\partial x^i} ,\frac{\partial}{\partial x^j}  \right) \colon \phi(U) \times \R^m \to \R$ for $1 \leq i,j \leq n$.

For $\lambda >0 $ we now consider the metrics $g_\lambda$ and $\tilde g_\lambda$ on $\frac{1}{\lambda} \phi(U) \times \R^m$ given by
\begin{equation} \label{scaleit}
 g_{\lambda} (x) := \sum_{i,j} g_{ij} ( \lambda x^1, \ldots, \lambda x^n ) \, dx^i dx^j  \, , \quad \tilde g_{\lambda} (x) := \sum_{i,j} g_{ij} ( \lambda x^1, \ldots, \lambda x^\ell , 0, \ldots, 0 )\, dx^i dx^j \, .
\end{equation}
We have an isometry
$$
\left( \frac{1}{\lambda} \phi(U) \times \R^m, g_\lambda \right) \stackrel{\alpha}{\approx} \left( E|_U , \frac{1}{\lambda^2} \, g \right), \qquad \alpha(x,y) := \Phi^{-1}(\lambda x, \lambda y).
$$
By \eqref{orthbas} and since $\Psi$ is an orthogonal bundle trivialization, we furthermore obtain an isometry
$$
 \left(\frac{1}{\lambda} \phi(U) \times \R^m, \tilde g_\lambda\right) \stackrel{\beta}{\approx} \left(U \times \R^m, \frac{1}{\lambda^2}\, g_B \oplus g_{\rm eucl.} \right), \qquad \beta(x,y) := \big( \phi^{-1} ( \lambda x) , y \big).
$$
Each mixed partial derivative of $g_{ij}\colon \phi(U) \times \R^m \to \R$ for  $1 \leq i,j \leq n$ is uniformly norm bounded over the compact set $\phi(K) \times D_1^m$ and hence the chain rule shows that
\begin{eqnarray} \label{limscal}
 \lim_{\lambda \to 0} \| g_{\lambda} - \tilde g_\lambda\|_{C_\lambda^2} = 0
\end{eqnarray}
where $C_\lambda^2$ denotes the  maximum $C^2$-norm over $\frac{1}{\lambda} \phi(K) \times D_1^m$ of smooth sections of $T^* V_\lambda \otimes T^* V_\lambda \to V_\lambda$  for $V_{\lambda} := \frac{1}{\lambda} \phi(U) \times \R^m$ and with respect to the frame $dx^i dx^j$.
Using that
\[
 \lim_{\lambda \to 0} \| \scal_{ \frac{1}{\lambda^2} g_B } \|_{C^0(B) } = 0,
\]
the compactness of $X$ and $\scal_h >0$, we conclude with \eqref{limscal} and the isometry $\beta$ that there exists $0 < \lambda_0 \leq 1$ such that for all $0 < \lambda \leq \lambda_0$ the metric $g_{\lambda}$ induces a metric on $\frac{1}{\lambda} \phi(U) \times X$ which is of positive scalar curvature on $\frac{1}{\lambda} \phi(K) \times X$.
Using the isometry $\alpha$ this shows that for all $0 < \lambda \leq \lambda_0$ the metric $\frac{1}{\lambda^2}g$, and hence also $g$, induce metrics on $\lambda V$ which are of positive scalar curvature on $\lambda V|_K$.

Since the compact manifold $B$ can be covered by finitely many such compact subsets $K$,  the assertion of Proposition \ref{tub_imm} follows.
\end{proof}

Given an immersion $f\colon M\to \R^N$, we denote by $\nu_f$ its normal bundle, whereas the normal bundle of an embedded submanifold $S\subset M$ will be denoted  $\NUMS$.

\begin{exa} Let $M$ be a Riemannian manifold, let $S \subset M$ be a closed submanifold of codimension at least $3$ and let $\rho_0 > 0$ such that the normal exponential map $\exp^{\perp} \colon \NUMS \to M$ restricts to a diffeomorphism $\{ |\eta| < \rho_0\} \approx U_{\rho_0}(S)$ of the open $\rho_0$-disc bundle in $\NUMS$ to the open $\rho_0$-neighborhood of $S$ in $M$.
Since $\exp^{\perp}$ induces a metric on $\{ |\eta| < \rho_0\} \subset \NUMS$ which is compatible with the Euclidean structure of $\NUMS \to S$ in the sense of  \dref{def:compatible}, \pref{tub_imm} implies that there exists $0 < \rho < \rho_0$ such that for all $0 < \rho' \leq \rho$ the induced metric on the normal spherical $\rho'$-tube $\exp^{\perp}(\{|\eta|=\rho'\})\subset M$ around $S$ has positive scalar curvature.

This statement also appears at the beginning of the proof of \cite{gl}*{Lemma 2} for trivial $\NUMS$.
Our proof of \pref{tub_imm} elaborates on the argument there.
\end{exa}

We will now apply \pref{tub_imm} to construct scalar-positive immersions of certain sub-fibre bundles of Euclidean vector bundles.

\begin{lem} \label{isom} Let $E_1, E_2 \to B$ be Euclidean vector bundles and let $\psi \colon E_1 \to E_2$ be an injective vector bundle homomorphism.
Then $\psi$ can be deformed through injective vector bundle homomorphisms into a fibrewise isometric vector bundle homomorphism $\psi'\colon E_1 \to E_2$.
\end{lem}

\begin{proof}
Let $r_1 \leq r_2$ be the ranks of $E_1$ and $E_2$, let $\Inj(r_1, r_2) \subset \R^{r_2 \times r_1}$ denote the space of matrices of maximal possible rank $r_1$ and let $\Iso (r_1, r_2) \subset \Inj (r_1, r_2)$ denote the subspace of matrices whose columns form an orthonormal family of vectors in $\R^{r_2}$.
The inclusion $\Iso(r_1, r_2) \subset \Inj (r_1 , r_2)$ is a strong deformation retract by the Gram-Schmidt process.
Hence the required deformation can be constructed inductively over a cellular decomposition of $B$ by standard obstruction theory.
\end{proof}

\begin{prop}[Normal bundle scaling] \label{postub}  In Setting \ref{assumption} suppose furthermore that $\scal_h > 0$ and that there exists an immersion $F \colon E \to \R^N$.
Then there exists a scalar positive immersion $f\colon V \to \R^N$.
\end{prop}

\begin{proof}
Consider the immersion  $\phi:=F|_B \colon B \to \R^N$ and let $\tau\colon B \times \R^N \to \nu_\phi$ be the fibrewise orthogonal projection onto the  normal bundle of $\phi$, considered as a subbundle of the trivial bundle $B \times \R^N \to B$.
Since $F$ is an immersion we obtain an injective vector bundle homomorphism $\psi\colon E \to \nu_\phi$ which over $q \in B$ is given by $\psi_q \colon E_q \stackrel{d_q F}{\longrightarrow} \R^N \stackrel{\tau_q}{\longrightarrow} (\nu_\phi)_q$.
By \lref{isom} we can deform $\psi$ into a fibrewise isometric vector bundle homomorphism $\psi' \colon E \to \nu_\phi$.

Next choose $\rho_0 > 0$ such that $\chi \colon \nu_\phi \to \R^N$, $\chi(q, \zeta) := \phi(q) + \zeta$, restricts to an immersion $\{|\eta| < \rho_0\} \to \R^N$.
Since  $\psi'$ is fibrewise isometric the metric on $\{ | \eta| < \rho_0 \} \subset E$ induced by $\chi \circ \psi'$  is compatible with the Euclidean structure on $E \to B$.

By \pref{tub_imm} we find $\lambda > 0$ with $\lambda V \subset \{ |\eta| < \rho_0 \} \subset E$ and such that the composition
$
f \colon V \stackrel{\lambda\,\cdot}{\longrightarrow} \lambda V \stackrel{ \chi \circ \psi'}{\longrightarrow} \R^N
$
is a scalar positive immersion.
\end{proof}

\begin{exa} \label{Veronese} Let $B$ be a closed $\ell$-dimensional manifold and let $V \to B$ be a fibre bundle with fibre $X = \CP^2$ and structure group $G = \UU(3) \rtimes \ZZ/2$ where $\ZZ/2$ acts by complex conjugation on $\UU(3)$ in the semidirect product.
As usual the group action of $G$ on $\CP^2$ is induced by the actions of $\UU(3)$ and $\ZZ/2$ on $\CC^3$ by left multiplication and complex conjugation, respectively.
This action is isometric for the Fubini-Study metric $g_{\rm FS}$ on $\CP^2$.
The total space $V$ of this bundle is of dimension $n = \ell + 4$.

We consider the affine subspace
\[
      {\rm H}_1(3,\CC):=\{ A \in \CC^{3\times 3} \mid A^* = A , \tr(A) = 1\}  \subset \{ A \in \CC^{3\times 3} \mid A^* = A \}
\]
with the Riemannian metric induced from the Euclidean  inner product  $\langle A, B \rangle := \tr(AB)$ on the right hand real vector space.
The map  $A \mapsto A - \frac{1}{3} \Id$ induces an isometry ${\rm H}_1(3, \CC) \approx (\R^8, g_{\rm eucl.})$ and hence an isomorphism between the group ${\rm Iso}\left({\rm H}_1 ( 3, \CC), \frac{1}{3}  \Id\right)$ of isometries of $ {\rm H}_1(3, \CC)$ fixing $\frac{1}{3} \Id$ and  the orthogonal group $\OO(8)$.

Now, as in \cite{ta68}*{(2.13)}, we consider the well-known Veronese isometric embedding
\[
   (  \CP^2, g_{\rm FS})  \hookrightarrow {\rm H}_1(3,\CC)  \approx ( \R^8, g_{\rm eucl.})
\]
which   is induced by the map $\CC^3 \supset S^5 \to {\rm H}_1(3,\CC)$,
\begin{equation} \label{explimmersion}
 (x_0, x_1, x_2) \mapsto \left( \begin{array}{ccc} |x_0|^2 & x_0 \overline{x_1} & x_0 \overline{x_2} \\
 x_1 \overline{x_0} & |x_1|^2 & x_1 \overline{x_2} \\
 x_2 \overline{x_0} & x_2 \overline{x_1} & |x_2|^2
 \end{array} \right) .
\end{equation}

The embedding \eqref{explimmersion} is equivariant with respect to the Lie group homomorphism $\psi \colon G \to {\rm Iso}\left({\rm H}_1 ( 3, \CC), \frac{1}{3}  \Id\right) \cong \OO(8)$ where $\psi(R,1)(A) := R A R^*$ and $\psi(R, -1)(A) := R \overline{A} R^*$ for $(R, \pm 1) \in \UU(3) \rtimes \ZZ/2$ and $A \in {\rm H}_1 ( 3, \CC)$.

Let $P \to B$ be the $G$-principal frame bundle of $V \to B$.
Setting $E := P \times_{\psi} \R^8$ we hence realize $V \to B$ as a sub-fibre bundle of $E \to B$ as described in Setting  \ref{assumption}.
By Cohen's Immersion Theorem \cite{co} applied to the sphere bundle of $E \oplus {\RR}$ there exists an immersion $E \to \R^{N}$ with $N = 2(\ell + 8) - \beta(\ell+8) = 2n + 8 - \beta(n+4)$ where $\beta(m)$ stands for the number of ones in the dyadic expansion of $m$.
 With  \pref{postub} we conclude that there exists a scalar positive immersion $V \to \R^{N}$.

A similar construction applies to fibre bundles $V \to B$ with fibre $X = \HP^2$ and structure group $G =\Sp(3) = \{ R \in \HH^{3\times 3} \mid R^* R = \Id\}$.
Formula \eqref{explimmersion} defines an isometric embedding $(\HP^2, g_{\rm FS}) \hookrightarrow {\rm H}_1(3,\HH):=\{ A \in \HH^{3\times 3} \mid A^* = A , \tr(A) = 1\} \approx ( \R^{14} , g_{\rm eucl.})$ which is equivariant with respect to the Lie group homomorphism $\psi\colon G \to \OO(14)$  where $\psi(R)(A) := RAR^*$.
Hence, in this case, we obtain a scalar positive immersion $V \to \R^N$ with $N = 2(\ell + 14) - \beta(\ell+14) = 2n + 12 - \beta(n+6)$ where $n = \dim V =  \ell + 8$.
\end{exa}

While these examples are the relevant ones for the proof of \tref{main} in \sref{extrsurg}, it is clear that the previous construction applies to the total spaces of many other fibre bundles.

\section{Local deformation I: Normally spherical immersions} \label{nomsph}

Most of the remaining parts of this paper will be devoted to the implementation of the extrinsic surgery process, following the spirit of \cite{gl}.
The purpose of this section is to show how to deform a scalar positive immersion into one which, near a closed submanifold, maps the normal discs of that submanifold to spherical caps  in Euclidean space.

\medskip

We first fix some notation.
Let $f\colon M \to \R^N$ be an immersion of an $n$-dimensional manifold.
Its differential $f_*=df$ identifies $TM$ with a subbundle of $f^*(T\R^N)\cong M \times\R^N$ whose orthogonal complement with respect to the Euclidean metric on $\R^N$ is the normal bundle $\nu_f $ of $f$.
We denote the induced fibre metrics and fibre norms on bundles constructed from $TM$ and $\nu_f$ by $\langle \,\cdot\, , \cdot \,\rangle_f$ and $| \,\cdot\, |_f$, where we suppress the subscript if the immersion~$f$ is obvious from the context.
We denote by $\alpha_f \in \Gamma(T^*M \otimes T^*M \otimes \nu_f)$ the second fundamental form of $f$.
Hence, $\tr(\alpha_f) \in \Gamma(\nu_f)$ is the (unnormalized) mean curvature field of $f$, while, by the Gauss equation,
\begin{equation}\label{scalformula}
\scal_f = | \tr(\alpha_f) |^2 -|\alpha_f|^2\colon M \to \R
\end{equation}
is the (unnormalized) scalar curvature of (the metric induced by) $f$.
In particular, if  $f$ is scalar positive, its mean curvature nowhere vanishes.
This fact was first pointed out in \cite{gr70}*{p.~42} and was used several times in the literature, see for example \cites{gu01, ta04}.
We obtain the unit normal field
\begin{equation} \label{xi}
\xi:=\tr(\alpha_f) /|\tr(\alpha_f)| \in \Gamma(\nu_f) \, .
\end{equation}
The field $\xi$ points in the direction along which $f$ will be deformed.
Intuitively, deforming $f$ in the direction of $\xi$ increases  the mean curvature faster than the second fundamental form (see the proof of Lemma \ref{comp_first_deform}), therefore increasing the scalar curvature by \eqref{scalformula}.

\medskip

 From now on assume that $f$ is scalar positive and let $S \subset M$ be a closed submanifold of codimension $k = n - \dim S$ and with normal bundle $\nu_S^M \to S$.
For $\rho > 0$ we set
\[
U_\rho (S) : = \{ p \in M \mid d( p,S ) < \rho \} \subset M,
\]
where $d$ refers to the induced Riemannian distance on $M$.
In this section we fix $\rho_0 > 0$ such that the normal exponential map $\exp^{\perp}\colon \NUMS \to M$ induces a diffeomorphism
\[
 \exp^{\perp} \colon \{ |\eta| < \rho_0 \} \stackrel{\approx}{\longrightarrow} U_{\rho_0} ( S) \, .
\]
Hence we write points $p\in U_{\rho_0}(S)$ in polar coordinates $(q,\omega, s)$ where $q\in S$, $\omega \in (\NUMS)_q$, $|\omega|_f = 1$, $s \in [0,\rho_0)$ and $p= \exp^{\perp}_q(s \, \omega)$.
Note that in these coordinates we have $q = (q,\omega,0)$ for all $q \in S$ and all such $\omega$.

We define smooth maps $F_{\tau} , G_{\tau} \colon U_{\rho_0} (S) \to \R^{N}$ by
\begin{eqnarray*}
 F_{\tau}(p) & :=& f(p) + \frac{1}{2} \, \tau \, s^2 \, \xi(q) \, \text{ for } \tau \geq 0 , \\
 G_{\tau} (p) & := & f(q) + \tau^{-1} \sin( \tau s) \, \omega + \tau^{-1}( 1 - \cos( \tau s))\, \xi(q) \, \text{ for } \tau > 0 \, .
\end{eqnarray*}
The map $G_\tau$ is smooth at $s=0$ since, for $e_{k+1}=(0,\dots,0,1)$ and
\[
    \hat S^k(1/ \tau) : = \{ |x - \tau^{-1}e_{k+1} | =\tau^{-1} \} \subset \R^{k+1},
\]
the map $\R^k \to \R^{k+1}$ defined in polar coordinates $(u, s) \in S^{k-1} \times [0, \infty)$ by
\[
(u, s) \mapsto \tau^{-1} \sin (\tau s) \, u + \tau^{-1} ( 1- \cos ( \tau s)) \, e_{k+1}
\]
can be interpreted as the (smooth) exponential map $\exp_0 \colon \R^k = T_0 \hat S^k(1/ \tau) \to \hat S^k(1/ \tau) \subset \R^{k+1}$.
For  $\tau > 0$,  $q \in S$ and $0 < \rho < \min \{\rho_0 , \frac{\pi}{2\tau} \}$ the map $G_\tau$ immerses the closed normal $\rho$-disc based at $q \in S$,
\[
     \{ (q, \omega, s) \mid \omega \in (\NUMS)_q , |\omega|_f = 1, s \leq \rho\}  \subset U_{\rho_0}(S),
\]
as a spherical cap in $\R^N$ based at $f(q)$ and opening in direction $\xi(q)$; see Figure \ref{figure:nsi}.

Observe that
\begin{equation}\label{einsjet}
 F_{\tau}|_S = f|_S = G_{\tau}|_S \quad \text{and} \quad dF_{\tau}|_S = df|_S= d G_{\tau}|_S\, .
 \end{equation}
Hence $F_{\tau}$ and $G_\tau$ restrict to immersions $U \to \R^N$ on some neighborhood $S \subset U \subset U_{\rho_0}(S)~\subset~M$, and the identifications of $TM|_S$ with a subbundle of $S \times \R^N$ coincide for the immersions $f$,~$F_\tau$~and~$G_\tau$.
The same holds for the normal bundles $\nu_f$, $\nu_{F_\tau}$ and $\nu_{G_\tau}$ restricted to~$S$.
In particular, the second fundamental forms of $F_{\tau}$ and $G_\tau$ restrict to smooth sections of $ T^*M|_S \otimes T^*M|_S \otimes (\nu_f)|_S \to S$.

\definecolor{darkblue}{rgb}{0.0, 0.2, 0.4}

\usetikzlibrary{shadings}
\begin{figure}
\begin{tikzpicture}[>=latex]
    \def\r{3}
    \def\H{1.5}
    \begin{scope}
     \clip
    ({\r*cos(180)},{\r*sin(180)}) arc [start angle=-180,end
    angle=0,radius=\r];
     \shade[top color=blue,bottom color=blue!50,opacity=0.6]
 ({-\r},{-1.1*\r}) rectangle +({2*\r},{0.1*\r+\H});
    \shade[top color =blue!20, bottom color = blue!40] (0,{-\r+\H}) circle [x radius={sqrt(\r^2-(\r-\H)^2)},
    y radius={0.2*sqrt(\r^2-(\r-\H)^2)}];
    \end{scope}
    \fill[fill=black] (0,-\r) circle (2pt);
    \draw node at (0,-\r -0.4) {$f(q)$} ;
    \draw[thick, dotted] (0,-\r) to (0, {- \r + \H - 0.2*sqrt(\r^2-(\r-\H)^2)})  ;
    \draw[thick, ->] (0,{-\r + \H- 0.2*sqrt(\r^2-(\r-\H)^2)}) -- (0,-0.5*\r) -- node[near end, right] {$\xi(q)$} (0,0);
\end{tikzpicture}
\caption{In blue: the immersion $G_{\tau}$ of  the normal disc at $q \in S$}
\label{figure:nsi}
\end{figure}

Our aim in this section is to prove in \pref{local_deform} below that for large $\tau$ the scalar positive immersion $f$ can be globally deformed, through scalar positive immersions, to bring it into the {\em normally spherical shape} $\,G_{\tau}$ near $S$.
This deformation will be constructed near $S$ by first applying the deformation $F_{t \tau}$, $t \in [0,1]$, which creates a large curvature contribution in the direction $\xi$, and then linearly interpolating between the resulting immersion and $G_\tau$.
Using the local flexibility lemma \cite{bh}*{Theorem 1} this local deformation near $S$ can be extended to the required global deformation of scalar positive immersions $M \to \R^N$.

\medskip

\pref{local_deform} essentially depends on the next three computational lemmas.
To state the first one, for $q \in S$ and $X \in T_q M$ let $X^{\top} \in T_q S$ and $X^{\perp} \in (\NUMS)_q$ denote the orthogonal projections.
Notice that these coincide for our three immersions $f$, $F_\tau$ and $G_\tau$ in view of \eqref{einsjet}.

\begin{lem} \label{deform} For all $q \in S$ and $X, Y \in T_q M$ we obtain that
\begin{eqnarray}
 \label{um} \alpha_{F_\tau} (X,Y) & = & \alpha_f(X,Y) + \tau \, \langle X^{\perp} , Y^{\perp} \rangle_f \, \xi (q) \, , \\
 \label{dois} \alpha_{G_\tau} (X,Y) & = & \alpha_f (X^{\top}, Y^{\top}) + \alpha_f (X^{\top} , Y^{\perp}) + \alpha_f( X^{\perp}, Y^{\top}) + \tau \, \langle X^{\perp}, Y^{\perp} \rangle_f \, \xi(q) \, .
\end{eqnarray}

\end{lem}

\begin{proof} First assume $X \in T_q S$, let $\beta \colon ( - \varepsilon, \varepsilon) \to S$ be a smooth curve through $q$ with $\beta'(0) = X$, and let $\hat Y \colon (-\varepsilon, \varepsilon) \to TM \subset M \times \R^N$ be a vector field along $\beta$ with $\hat Y(0) = Y$.
By \eqref{einsjet} both $\alpha_{F_\tau} (X,Y)$ and $\alpha_{G_\tau} (X,Y)$ are equal to the orthogonal projection of $\hat Y'(0) \in \R^N$ onto $(\nu_f)_q$ and hence are equal to $\alpha_f(X,Y)$.
This and the symmetry of second fundamental forms show that for proving \lref{deform} we can restrict to the case $X,Y \in (\NUMS)_q$, and by polarization and bilinearity we can further restrict to the case $X = Y = \omega \in (\NUMS)_q$, $|\omega| = 1$.

Let $\beta\colon(-\varepsilon,\varepsilon) \to (\NUMS)_q \subset \R^N$ be the curve $\beta(s):=s\, \omega$.
Then $\beta'(0) = \omega$, and $(F_\tau\circ\beta)''(0)=(f\circ \beta)''(0)+ \tau \, \xi(q)$ and $(G_\tau \circ \beta )''(0)=\tau \, \xi(q)$.
This gives \eqref{um} and \eqref{dois} after projection onto $(\nu_f)_q$.
\end{proof}

\begin{lem} \label{comp_first_deform} Along $S$ we have  $ \scal_{F_{ \tau}} > 0$ for all $\tau \geq 0$.
\end{lem}

\begin{proof} We work along $S$ throughout.
Let $\tau \geq 0$.
As $\tr(\alpha_f)$ is a positive multiple of $\xi$, \eqref{um} implies that
\[
 |\tr(\alpha_{F_\tau} )|^2 = \big( |\tr(\alpha_f)|+ \tau k \big)^2 = |\tr(\alpha_f)|^2+ 2 |\tr(\alpha_f)| \tau k + \tau^2 k^2 \, .
\]
Furthermore, by the triangle inequality,
\[
 |\alpha_{F_\tau} |^2 \leq \big( | \alpha_f|+ \tau \sqrt{k} \big)^2 = |\alpha_f|^2 + 2 |\alpha_f| \tau \sqrt{k} + \tau^2 k \, .
\]
Since $|\tr(\alpha_f)| > | \alpha_f|$ by our assumption $\scal_f > 0$, the Gauss equation gives us
$$
 \scal_{F_\tau} = |\tr(\alpha_{F_\tau} )|^2 - | \alpha_{F_\tau} |^2 \geq \scal_f + \tau^2 (k^2 - k) \geq \scal_f > 0 \, .
$$
\vskip -0.7cm\end{proof}\vskip 0.15cm

\begin{lem}\label{affine_comb}
If $k \geq 2$,  there exists $\tau_0 > 0$ such that, for all $\tau \geq \tau_0$ and $t \in [0,1]$, it holds that $\scal_{(1-t) F_\tau + t G_\tau} > 0$ along $S$.
\end{lem}

\begin{proof} By \eqref{um} and \eqref{dois} there exists $C \geq 0$, which only depends on the restriction of $\alpha_f$ to $S$, such that, for all $\tau > 0$, $q \in S$, and $X,Y\in T_qM$, we get
\[
 \big| (1- t) \, \alpha_{F_\tau} (X,Y) + t \, \alpha_{G_\tau} (X,Y) - \tau \langle X^{\perp} , Y^{\perp} \rangle \, \xi(q) \big| \leq C |X| |Y| \, .
 \]
Hence, by the triangle inequality,
\[
 \big| \tr\big( (1- t) \, \alpha_{F_\tau} + t \, \alpha_{G_\tau} \big) \big| \geq k \tau - nC \, .
\]
Similarly,
\[
 \big| (1- t) \, \alpha_{F_\tau} + t \, \alpha_{G_\tau} \big| \leq \sqrt{k} \tau + \sqrt{n^2 C^2} = \sqrt{k} \tau + nC \, .
\]
Assuming that $k \tau \geq nC$, the Gauss equation hence implies that, along $S$,
$$
\scal_{(1-t) F_\tau + t G_\tau} \geq (k\tau - nC)^2 - ( \sqrt{k} \tau + nC)^2 = ( k^2 - k) \tau^2 - 2( k + \sqrt{k}) \tau n C.
$$
Since $k \geq 2$, there exists $\tau_0 \geq nC / k$ such that the last expression is positive for all $\tau \geq \tau_0$.
\end{proof}

We finally have all the ingredients to prove the main result of this section.

\begin{prop}[Normally spherical immersions] \label{local_deform}
If $k \geq 2$, there exists $\tau_0 > 0$ such that for all $\tau \geq \tau_0$ there exist $0 < \rho \leq \rho_0$ and a continuous family $f_t\colon M \to \R^N$, $t \in [0,1]$, of scalar positive immersions with
$f_0 = f$, ${f_t}|_{ M \setminus U_{\rho_0}(S)} = f|_{ M \setminus U_{\rho_0}(S)}$ for $t \in [0,1]$ and
$$
{f_{t}}|_{U_\rho(S)} =
\begin{cases}
F_{2t\tau}|_{U_\rho(S)} & \text{ for } 0 \leq t \leq 1/2 , \\
\big( (2-2t) F_\tau + (2t-1) G_\tau \big) |_{U_\rho(S)} & \text{ for } 1/2 \leq t \leq 1.
\end{cases}
$$
In particular, ${f_1}|_{ M \setminus U_{\rho_0}(S)} = f|_{ M \setminus U_{\rho_0}(S)}$ and $f_1|_{U_{\rho}(S) } = G_{\tau}|_{U_{\rho}(S)}$.
\end{prop}

\begin{proof} Choose $\tau_0$ as in \lref{affine_comb} and let $\tau \geq \tau_0$.
By \eqref{einsjet}, \lref{comp_first_deform} and \lref{affine_comb}, there exists an open neighborhood $S \subset U \subset U_{\rho_0}(S) \subset M$ such that for all $t \in [0,1]$ the maps $F_{t\tau}$ and $(1-t) F_\tau + t G_\tau$ restrict to scalar positive immersions $U \to \R^N$ whose $1$-jets along $S$ do not depend on $t$.
Since being a scalar positive immersion defines an open partial differential relation on the $2$-jets of maps $M \to \R^N$, the claim follows from the local flexibility lemma \cite{bh}*{Theorem 1}.
\end{proof}

With the help of Gromov's $h$-principle for open, Diff-invariant partial differential relations over open manifolds, see \cite{gr86}, the computations in this section can also be used to justify Remark~\ref{remquest} as follows.

\begin{prop} \label{hprinc}
Let $M$ be a non-compact connected manifold of dimension at least $2$ admitting an immersion $M \to \R^N$.
Then there exists a scalar positive immersion $M \to \R^{N+1}$.
\end{prop}

\begin{proof}
Consider the trivial vector bundle $X = M \times \R^{N+1} \to M$ and the bundle $X^{(2)} \to M$ of $2$-jets of smooth maps $M \to \R^{N+1}$.
Given a smooth map $f \colon M \to \R^{N+1}$, we denote by $j^2 f \colon M \to X^{(2)}$ its second order jet map.
Recall that for $p \in M$ the value $j^2 f(p) \in (X^{(2)})_p$ only depends on the restriction of $f$ to some neighborhood of $p$.
Being a scalar positive immersion  defines an open, ${\rm Diff}(M)$-invariant partial differential relation $\mathscr{R} \subset X^{(2)} $.

Let $\phi \colon M \to \R^N$ be an immersion.
For a continuous map $\tau \colon M \to \R$ consider the continuous section $\phi_{\tau} \colon M \to X^{(2)}$,
\[
p \mapsto \left(j^2 \phi(p), \tau(p)\, j^2\big( x \mapsto \, d(p,x)^2 \big) (p)\right).
\]
Since $\dim M \geq 2$, a computation as in the proof of Lemma \ref{comp_first_deform} implies that, for each compact $K \subset M$, there exists $\tau_0 \in (0, \infty)$ such that, if $\tau \geq \tau_0$ on $K$, we have $\phi_{\tau}(K) \subset \mathscr{R}$, that is $\phi_\tau$ formally solves $\mathscr{R}$ over $K$.
Using a locally finite cover of $M$ by relatively compact open subsets, we hence find $\tau\colon M \to \R$ such that $\phi_{\tau}$ formally solves $\mathscr{R}$ over $M$.
Gromov's $h$-principle implies that there exists a smooth map $f \colon M \to \R^{N+1}$ solving $\mathscr{R}$.
\end{proof}

\section{Local deformation II: Bending profiles} \label{bending}

The aim of this section is to show that a scalar positive immersion which is normally spherical near a closed submanifold as in \pref{local_deform} can be further deformed, again through scalar positive immersions, into a shape proper to add a surgery handle.

\medskip

As in the previous section, let $f \colon M \to \R^N$ be a scalar positive immersion, let $n := \dim M$ and let $S \subset M$ be a closed embedded submanifold of codimension $k$ and with normal bundle $\NUMS \to S$.
If $E \to B$ is a Euclidean vector bundle and $\rho > 0$, we denote by $D_{\rho}(E) = \{  |\eta| \leq \rho\} \to B$ the closed $\rho$-disc bundle and by $S_{\rho}(E)= \{  |\eta| = \rho\} \to B$ the $\rho$-sphere bundle of $E$.
Points in $S_{1}(\NUMS)_q$ are written in the form  $(q, \omega)$ with $\omega \in (\NUMS)_q$ of norm one.

Since $S$ is compact and $\xi$ in \eqref{xi} is normal to $f$, we find $0 < \rho_0 \leq 1$ such that the map $S_1(\NUMS) \times D_{\rho_0 } (\R^2) \to \R^N$,
\[
 (q,\omega, a,b) \mapsto f(q) + a \, \omega + b \, \xi(q) \, ,
\]
is an immersion.
In the remainder of this section we fix such a $\rho_0$.

\medskip

\begin{dfn} Let $I \subset \R$ be a compact interval and $\gamma\colon I \to \R^2$,  $\gamma(s) = (a(s) ,b(s))$, be a regular smooth curve.
For $0 < \rho \leq \rho_0$, we say that $\gamma$ is of {\em extent} $\rho$, if $| \gamma(s)| < \rho$ for all $ s\in I$.
\end{dfn}

We now consider the compact manifold with boundary
\[
 \Sigma := S_1(\NUMS) \times I .
\]
For $\gamma$ of extent $0 < \rho \leq \rho_0$ we obtain an immersion $F_\gamma\colon \Sigma \to \R^N$ along the {\em bending profile} $\gamma$,
\begin{equation} \label{defFg}
 F_{\gamma} (q, \omega, s) := f(q) + a(s) \, \omega + b(s) \, \xi (q) \, .
\end{equation}

In this section we will first derive a lower bound for $\scal_{F_\gamma}$ for certain $\gamma$; see \pref{main_est}.
This requires some preparation which we shall again split into a number of lemmas.
After solving a pertinent ODE for $\gamma$ in  Lemma \ref{bc}, Proposition \ref{bending_curve}  provides the bending profiles required for the extrinsic surgery in  \sref{extrsurg}.

\medskip

The projection $\pi \colon \Sigma \to S$, $ \pi (q, \omega, s) := q$,  is a smooth submersion and hence induces an orthogonal direct sum decomposition of $T\Sigma$ into vertical and horizontal subbundles,
\[
 \mathscr{V} = \ker d\pi \subset T \Sigma \, , \quad \mathscr{H} = \mathscr{V}^{\perp_{F_{\gamma} }} \subset T\Sigma \, .
\]
For $X \in T\Sigma$ we denote by $\mathscr{V}X \in \mathscr{V}$ and $\mathscr{H}X \in \mathscr{H}$ its vertical and horizontal components.
Note that for $p = (q, \omega, s) \in \Sigma$ we have an orthogonal splitting of $\mathscr{V}_p$ with respect to the metric induced by $F_\gamma$,
\begin{equation} \label{splittingtan}
 \mathscr{V}_p = \omega^{\perp} \oplus {\rm span} \{\partial_s\} \subset (\NUMS)_q \oplus^{\perp} T_s I \, .
\end{equation}

Let $K \subset S$ be a compact subset which is contained in some open coordinate neighborhood $K \subset U\subset S$ admitting an orthogonal local bundle trivialization
\[
 \Psi \colon \NUMS|_U\stackrel{\cong}{\to} U\times \R^k \, .
\]
This induces a diffeomorphsim $\pi^{-1}(U) \approx U\times \skmu \times I$.
For $p = (q, \omega, s) \in \Sigma$ with $q \in U$ and $\omega \in \skmu \approx S_1(\NUMS)_q$, we hence obtain a direct sum decomposition
\begin{equation} \label{splitting}
 T_p \Sigma \cong T_q S \oplus T_\omega \skmu \oplus T_s I
\end{equation}
with $T_{\omega} \skmu = \omega^{\perp} \subset \R^k$.
Note that $T_q S$ is, in general, not orthogonal to $\mathscr{V}_p$.
For $X \in T_q S$ we denote by $X_p \in T_p \Sigma$ the vector $(X, 0,0)$ in the decomposition \eqref{splitting}.

\begin{lem} \label{difficult}
For all such $K$ and $\Psi$ there exists $C \geq 0$ such that, for all unit speed curves $\gamma$ of extent $0 < \rho \leq \rho_0$, for all $p = (q, \omega, s) \in K \times \skmu \times I$ and for all $X,Y \in T_q S$, $V \in \mathscr{V}_p$ and  $Z \in (\NUMS)_q \oplus \R \xi(q) \subset \R^N$, it holds that
\begin{eqnarray}
 \label{deux}
 \big| \alpha_{F_{\gamma}} (X_p ,Y_p) \big|_{F_\gamma} & \leq & C |X|_f |Y|_f  , \\
 \label{trois}
 \big| \alpha_{F_{\gamma}} (X_p, V \big)\big|_{F_\gamma} & \leq & C |X|_f |V|_{F_{\gamma}}  , \\
 \label{quatre}
 \big| \langle d_p F_{\gamma} (X_p) , Z \rangle \big| & \leq & \rho C |X|_f |Z| , \\
 \label{estvert}
 \big| \mathscr{V}X_p \big|_{F_\gamma}  & \leq & \rho C |X|_f , \\
 \label{esthor}
\big|\mathscr{H} X_p\big|_{F_\gamma} & \geq & (1 - \rho C) |X|_f .
 \end{eqnarray}
 \end{lem}

\begin{proof}
For $\eta \in \R^k$ we define $\hat \eta \colon U\to \R^N$ as $\hat \eta (q) := \Psi_q^{-1} ( \eta ) \in (\NUMS)_q \subset \R^N$.
Hence the standard basis $(e_1, \ldots, e_k)$ of $\R^k$ yields an orthonormal frame $(\hat e_1, \ldots, \hat e_k)$ of $(\NUMS)|_U$.
Choose local coordinates $(x^1, \ldots, x^{n-k})$ over $U$.
Setting $\partial_i = \frac{\partial}{\partial x^{i}}$ for $1 \leq i \leq n-k$ this induces a local frame $(\partial_1, \ldots, \partial_{n-k})$ of $TS$ over $U$.

In each of the following estimates, $C$ denotes some non-negative constant which depends on the local coordinates $(x^{1}, \ldots, x^{n-k})$, on  the restriction of the metric tensor on $S$ to $K$ and on the $2$-jets over $K$ of the $\R^N$-valued smooth functions $f$, $\xi$ and $\hat e_1, \ldots, \hat e_k$, but not on $\gamma$.

First let $X = \partial_i$ and $Y = \partial_j$ for  $1 \leq i,j \leq n-k$.
Since $a$ and $b$ are norm bounded by $1$ (recall $\rho_0 \leq 1$) and $| \omega|_f = 1$, we obtain
\[
 | \alpha_{F_{\gamma}} (X_p, Y_p) |_{F_\gamma} \leq | \partial_i \partial_j F_{\gamma} (q , \omega, s) | \leq | \partial_i \partial_j f (q) | + |a(s)| | \partial_i \partial_j \hat \omega (q) | + |b(s)| | \partial_i \partial_j \xi (q) | \leq C .
\]
Together with  the bilinearity of $\alpha_{F_\gamma}$ this implies that for all $X ,Y \in T_q S$ we get
\[
     | \alpha_{F_{\gamma}} (X_p, Y_p) |_{F_\gamma} \leq C  |X|_f |Y|_f
\]
which is \eqref{deux}.

Next let $X = \partial_i$ for $1 \leq i \leq n-k$ and $V \in T_{\omega } \skmu = \omega^{\perp} \subset \R^k$, which we consider as a vector in $T_p\Sigma$ by \eqref{splitting}.
Note that $\partial_{V} F_{\gamma} (p) = a(s) \hat V (q)$.
Writing $V = \sum_{j=1}^k \alpha_j e_j$ with $\alpha_j \in \R$  this gives
\[
 | \alpha_{F_{\gamma}} (X_p,V ) |_{F_\gamma} \leq |a(s)| | \partial_i \hat V (q) | \leq  |a(s)| \sum_{j=1}^k | \alpha_j| |\partial_i \hat e_j(q)|  \leq C  |a(s)|   |V|_f =  C |V|_{F_\gamma} .
\]
Moreover, for $V = \partial_s \in T_s I$ we get, using that $\gamma$ is of unit speed and $|\omega|_f = 1$,  that
\[
 | \alpha_{F_{\gamma}} (X_p,V ) |_{F_\gamma} \leq | \partial_s \partial_i F_{\gamma}(q , \omega, s  ) | = |a' (s) | | \partial_i \hat \omega(q)  | + | b'(s)| | \partial_i \xi(q) | \leq C .
\]
For $ V \in \mathscr{V}_p \cong T_\omega \skmu \oplus^{\perp} T_s I $ the last two  estimates  imply
\[
    | \alpha_{F_{\gamma}} (X_p,V ) |_{F_\gamma} \leq C |V|_{F_\gamma}
\]
such that, for all $X \in  T_q S$ and $V \in \mathscr{V}_p$, we have
\[
      | \alpha_{F_{\gamma}} (X,V ) |_{F_\gamma} \leq C |X|_f |V|_{F_\gamma}
\]
which is \eqref{trois}.

For $X = \partial_i$, $1 \leq i \leq n-k$, we obtain
\[
    |d_pF_\gamma(X_p) - d_q f (X) | \leq  |a(s)| | \partial_i \hat \omega(q) | + | b(s)| | \partial_i \xi(q) |  \leq \rho C
\]
such that for all $X \in T_q S$ we get
\begin{eqnarray} \label{masterest}
    |d_pF_\gamma(X_p) - d_q f (X) |   \leq \rho C|X|_f  .
\end{eqnarray}
Since  $d_q f(X) \perp Z$ estimate \eqref{masterest} and the Cauchy-Schwarz inequality imply
\[
 \langle d_p F_{\gamma} (X_p) , Z \rangle \leq  |d_pF_\gamma(X_p) - d_q f (X) | |Z| \leq \rho C |X|_f |Z|
 \]
which is \eqref{quatre}.

Finally, since $d_q f(X) \perp d_p F_\gamma(\mathscr{V}_p)$ estimates \eqref{estvert} and \eqref{esthor} follow from \eqref{masterest} by projecting $d_pF_\gamma(X_p) - d_q f (X)$ onto $d_pF_\gamma(\mathscr{V}_p)
\subset \R^N$ and $d_pF_\gamma(\mathscr{V}_p)^{\perp} \subset \R^N$, respectively.
\end{proof}

If $\gamma$ is of unit speed we define the unit vector field $N \colon \Sigma \to S_1 \big(\NUMS \oplus \R \xi \big)$ by
\[
   N(q, \omega,s) := -b'(s) \omega + a'(s) \xi(q) \in \R^N ,
\]
and  decompose it into orthogonal summands as
\[
 N = N^\top+ N^\perp\ \in\ dF_{\gamma}(T\Sigma)\, \oplus^\perp \nu_{F_{\gamma}} = F_{\gamma}^*(T\R^N).
\]
For $q \in S$ we define $\Sigma_q := \pi^{-1}(q) = \{ q\} \times S_1(\NUMS)_q \times I \subset \Sigma$, and observe that $F_\gamma$ restricts to an embedding
\[
 \Sigma_q \hookrightarrow (\NUMS)_q \oplus \R \, \xi(q)
\]
whose image is the revolution hypersurface with meridian $\gamma$ and axis $\R \, \xi(q)$.
This embedding has $N_q (\omega, s) :=  -b'(s) \omega + a'(s) \xi(q) = N(p)$ as unit normal vector field, that is to say the Gauss map.

\begin{lem} \label{estnormal} There exists $0 < \rho \leq \rho_0$ such that, for all $\gamma $ of unit speed and of extent $\rho$, we have $|N^{\perp}|\geq 1/2$.
\end{lem}

\begin{proof}
We first work in the setting of Lemma \ref{difficult}.
For $p = (q, \omega, s) \in \Sigma$ with $q \in K$ we get  $N(p) \perp dF_{\gamma} (\mathscr{V}_p)$ and hence
\begin{equation*} \label{esttan}
 |N^\top(p)| = \max_{X \in T_q S, | \mathscr{H}X_p |_{F_\gamma} = 1}    \langle dF_{\gamma}( X_p) , N(p) \rangle  .
\end{equation*}
Pick $X \in T_q S$ for which this maximum is attained.
By \eqref{quatre} and \eqref{esthor} we have $\langle dF_{\gamma} (X_p) , N(p) \rangle \leq \rho C | X |_f$ and $( 1 - \rho C) | X |_f \leq | \mathscr{H} X_p|_{F_\gamma} = 1$.
We therefore find $0 < \rho \leq \rho_0$ such that for all $\gamma$ of extent $\rho$ and all such $p$ we have $|N^\top(p)|\leq 1/2$, and hence $| N^{\perp}(p) | \geq 1/2$.

Since the compact manifold $S$ can be covered by finitely many $K$ to which Lemma \ref{difficult} applies, the assertion of \lref{estnormal} follows.
\end{proof}

Given a unit speed curve $\gamma = (a,b) \colon I \to \R^2$ satisfying $a(s) \neq 0$ for all $s \in I$ we define the smooth functions $\kappa, \sigma\colon I \to \R$ by
\begin{equation} \label{kappasigma}
 \kappa := a' \, b'' - a'' \, b' \, , \quad \sigma := b' / a \, .
\end{equation}
Notice that $\kappa$ is the curvature of $\gamma$ with respect to its unit normal $(-b', a') \in \R^2$.

\begin{lem} \label{vertical} With respect to the direct sum decomposition $\mathscr{V}_p = \omega^{\perp} \oplus  {\rm span} \{ \partial_s \}$ we obtain
\[
 (\alpha_{F_{\gamma}} )|_{\mathscr{V}_p \times \mathscr{V}_p } = \big( \sigma(s) \, \langle \,\cdot\,, \cdot \,\rangle_{\omega^{\perp} } + \kappa(s) \, ds^2\big) \, N^{\perp}(p) \, .
\]
\end{lem}
\begin{proof} A direct computation shows that the second fundamental form $ \alpha_q$ of the embedding $\Sigma_q \hookrightarrow (\NUMS)_q \oplus \R \, \xi(q)$ is given, with respect to the orthogonal decomposition $T_{(\omega, s)} \Sigma_q = \omega^{\perp} \oplus {\rm span} \{ \partial_s \}$, by
\[
 \alpha_q = \big( \sigma(s) \, \langle \,\cdot\,, \cdot\, \rangle_{\omega^{\perp} } + \kappa(s) \, ds^2\big) \, N(p) \, .
\]
The assertion now follows from the definition of $N^{\perp}(p)$.
\end{proof}

\begin{dfn} A smooth curve $\gamma\colon I \to \R^2$ is called {\em controlled}, if it is of unit speed, $a(s) \neq 0$ for all $s \in I$ and $\frac{2-k}{4} \, \sigma \leq \kappa \leq \sigma$ on $I$.
(Recall that $k$ is the codimension of $S$ in $M$.)
\end{dfn}

\begin{rem} \label{contr} If $\gamma$ is controlled, we have $\sigma \geq 0$ and $\max \{ |\kappa| , \sigma \} \leq n \sigma$.
\end{rem}

For $p = (q, \omega, s) \in \Sigma$ and $X \in T_q S$ we denote by $\mathscr{H}_p X \in \mathscr{H}_p$ the unique horizontal tangent vector satisfying $d_p \pi(\mathscr{H}_p X) = X$.
Note that whenever we work in a decomposition \eqref{splitting}, we have $\mathscr{H}_p X = \mathscr{H}X_p$.
In particular the horizontal component of $X_p$ is independent from the chosen bundle trivialization $\Psi$.

\begin{lem} \label{estsecondfund} There exist constants $C \geq 0$ and $0 < \rho \leq \rho_0$ with the following property:
If $\gamma$ is controlled and of extent $\rho$, then, for all $p = (q, \omega, s) \in \Sigma$, $X, Y \in T_q S$ and $V \in \mathscr{V}_p$, we have
\begin{eqnarray}
 \label{oans}
 \big| \alpha_{F_{\gamma}} (\mathscr{H}_p X, \mathscr{H}_pY ) \big|_{F_{\gamma} } & \leq & C \big(1 + \rho \sigma \big) \, |X|_f \, |Y|_f \, , \\
 \label{zwoa}
 \big| \alpha_{F_{\gamma}} (\mathscr{H}_p X, V) \big|_{F_\gamma} & \leq & C \big( 1 + \rho \sigma \big) \, |X|_f \, |V|_{F_\gamma} \, .
\end{eqnarray}
 \end{lem}
\begin{proof}
Again it is enough to work in the setting of Lemma \ref{difficult}.
In the following we replace the constant $C$ appearing in Lemma \ref{difficult} by $\max \{ C, 1 \}$.

By \eqref{deux}, \eqref{trois}, \eqref{estvert}, Lemma \ref{vertical}  and Remark \ref{contr} we obtain, using  $\mathscr{H}_p X = X_p - \mathscr{V} X_p$, $\mathscr{H}_p Y = Y_p - \mathscr{V} Y_p$ and  $0 < \rho \leq 1$,
\begin{align*}
 \big| \alpha_{F_{\gamma}} (\mathscr{H}_p X, \mathscr{H}_p Y) \big|_{F_\gamma} & \leq C \left( |X|_f |Y|_f +  |X|_f |\mathscr{V} Y_p|_{F_\gamma} + |\mathscr{V} X_p|_{F_\gamma} |Y|_f + n\sigma |\mathscr{V} X_p|_{F_\gamma} | \mathscr{V} Y_p |_{F_\gamma}\right) \\
    & \leq C \left( 1 + 2\rho C + n\sigma \rho^2 C^2 \right) |X|_f |Y|_f \\
    & \leq C^3 \left(  1  + 2  + n \rho \sigma \right) |X|_f |Y|_f .
\end{align*}
Hence we get \eqref{oans} with $C$ replaced by $3 nC^3$.
Estimate \eqref{zwoa} is implied in an analogous fashion by
\[
 \big| \alpha_{F_{\gamma}} (\mathscr{H}_p X, V ) \big|_{F_\gamma} \leq C\big( |X|_f |V|_{F_{\gamma}} + n \sigma |\mathscr{V} X_p|_{F_\gamma} |V|_{F_{\gamma}}\big) \leq C( 1 + n \sigma \rho C) |X|_f |V|_{F_\gamma} \, .
\]
\vskip -0.7cm\end{proof}\vskip 0.01cm

\begin{prop} \label{main_est} Assume $k \geq 3$.
Then  there exist constants $C \geq 0$ and $0 < \rho \leq \rho_0$ with the following property:
If $\gamma$ is controlled and of extent $\rho$, then
\begin{equation} \label{scalest}
 \scal_{F_\gamma} \geq \frac{(k-1)( k-2)}{16} \sigma^2 - C \sigma - C \, .
\end{equation}
In particular, there exists a constant $\sigma_0 > 0$ such that $F_\gamma$ is scalar positive for all such $\gamma$ satisfying   $\sigma \geq \sigma_0$.
\end{prop}
\begin{proof}
With respect to the orthogonal direct sum decomposition $T_p \Sigma = \mathscr{V}_p \oplus \mathscr{H}_p$ write
\[
 \alpha_{F_{\gamma}} = \left( \begin{array}{cc} \Delta & B \\ B^T & Q \end{array} \right) \, ,
\]
where $\Delta := (\alpha_{F_{\gamma}} )|_{\mathscr{V}_p \times \mathscr{V}_p }$ was computed in \lref{vertical}.
The Gauss equation  hence implies that
 \begin{align*}
 \scal_{F_\gamma} (p) & = | \tr( \alpha_{F_{\gamma}} ) |^2 - | \alpha_{F_{\gamma}} |^2 \\
 & = \big|\tr( \Delta) + \tr(Q)\big|^2 - |\Delta|^2 - 2 |B|^2 - |Q|^2 \\
 & \geq \left( |\tr(\Delta)|^2 - |\Delta|^2\right) - 2|\tr(\Delta)||\tr(Q)| + |\tr(Q)|^2 - 2|B|^2 - |Q|^2 \, .
 \end{align*}
By \eqref{esthor} and \lref{estnormal} we find $0 < \rho \leq \rho_0$ such that, for all $\gamma$ of extent $\rho$, $p = (q, \omega, s) \in \Sigma$ and $X \in T_q X$, we have
\begin{equation} \label{standest}
 | \mathscr{H}_pX |_{F_\gamma} \geq |X|_f/2 \, , \quad 1/2 \leq | N^{\perp}(p)| \leq 1 \, .
\end{equation}
Since $\gamma$ is controlled, we get $2 \kappa \geq - \frac{k-2}{2} \, \sigma$ and hence
\[
 |\tr(\Delta)|^2 - |\Delta|^2 = (k-1) ( (k-2) \sigma^2 + 2 \kappa \sigma ) |N^{\perp}(p)|^2 \geq \frac{(k-1)(k-2)}{8} \sigma^2 \, .
\]
Using \eqref{standest}, Remark \ref{contr} and \lref{estsecondfund} we see that the entries of $Q \in ( \R^{N})^{(n-k) \times (n-k)}$ are norm bounded by $ 4 C ( 1+ \rho\sigma) $, the ones for $B \in ( \R^{N})^{k \times (n-k)}$ are norm bounded by $ 2 C ( 1 + \rho \sigma)$ and the ones for $\Delta$ are norm bounded by $n\sigma$.
Hence $2 |\tr(\Delta)||\tr(Q)| + 2|B|^2 + |Q|^2$ is bounded by a quadratic polynomial in $\sigma$, and passing to a smaller $\rho$ we can assume that the coefficient of $\sigma^2$ is bounded by $ \frac{(k-1)(k-2)}{16}$, which is positive as $k \geq 3$.
This completes the proof of \pref{main_est} for an appropriate $C$.
\end{proof}

\begin{lem} \label{bc} Let $k \geq 3$, $(x,y) \in \R^2$ with $x > 0$ and $(u, v) \in S^1$ with $u, v > 0$.
Set $\lambda := \frac{k-2}{4} > 0 $.
Then there exists $- \frac{\pi \, x}{2\lambda v} < R < 0 $ and a unit speed curve $\gamma = (a,b) \colon [R,0] \to \R^2$ with the following properties:
\begin{enumerate}[(i)]
\item \label{xx} $\gamma(0)=(x, y)$ and $a(s) > 0$ for all $s \in [R,0]$;
\item \label{aa} $\gamma'(R)=(0, 1)$ and $\gamma'(0) = (u,v)$.
 In particular $\sigma(0) = v/x$;
\item \label{bb} $\kappa = - \lambda\, \sigma$ with $\kappa$ and $\sigma$ as in \eqref{kappasigma}.
\end{enumerate}
\end{lem}
\begin{proof}
Consider a maximal solution $\gamma \colon I \to \R^2$, $\gamma(s)=(a(s), b(s))$, $0 \in I \subset \R$, of the system of second order nonlinear ordinary differential equations
\begin{equation}\label{ODE}
\left(\begin{array}{c}a''\\b''\end{array}\right)=-\lambda\, \frac{b'}{a}\left(\begin{array}{c}-b'\\a'\end{array}\right)
\end{equation}
with initial conditions $\gamma(0) = (x,y)$ and $ \gamma'(0) = (u, v)$.
Then $\langle \gamma' , \gamma'' \rangle = 0$ and hence $\gamma$ has unit speed.
Furthermore, the quantity $z(s) \colon=b'(s)a(s)^{\lambda}$ is preserved along $ \gamma$ since
\[
z' = b'' a^{\lambda}+ \lambda b' a^{\lambda - 1} a' =- \lambda  \frac{b'}{a}  a' a^\lambda+\lambda b' a' a^{\lambda-1} = 0 \, .
\]
Therefore $z(s) = z(0) > 0$ for all $s$.
Since $a(0) = x > 0$ and $|b'| \leq 1$ this implies that $a$ is bounded below by a constant $C > 0$, and hence $b' > 0$ on $I$.
In particular  $(a, a', b')$ stays in $[C, \infty) \times [-1, 1] \times [0 ,1]$ and hence $I = \R$.

As $b' > 0$ we obtain a continuous function $\theta \colon \R \to (0,\pi)$ which measures the angle in counterclockwise direction between $(1,0) \in \R^2$ and $ \gamma'(s)$, that is, $\cos(\theta)= a'$ and $\sin(\theta)=b'$.
Moreover, we have
\[
 \theta'= \kappa= - \lambda \sigma \, ,
\]
where the second equality uses the Frenet equation and \eqref{ODE}.
Since $z$ is constant along $ \gamma$ and $b' > 0$, we know that $b'$ and hence $\sigma = b' / a$ are decreasing on the subset $\{a'>0\} \subset \R$.
Combining this with $\sigma(0)= b'(0)/a(0)= v/x$ we conclude that $\theta' \leq - \frac{\lambda v}{x}$ on the maximal interval $(R,0]$, $R < 0$, on which $a'>0$.

Since $a'(0) = u > 0$ and $\theta(0)\in(0,\pi/2)$, we get $-\frac{\pi \, x}{2\lambda v} < R < 0$ and $\theta(R) = \pi/2$.
This implies $a'(R) = 0$ and hence $b'(R) = 1$ since $\gamma$ has unit speed.
\end{proof}

\begin{rem} This proof is inspired by \cite{ef}*{Lemma 3.14}, but we preferred to solve a differential equation for $\gamma$ instead of writing $a= h(b)$ and solving a differential equation for $h$.
\end{rem}

\begin{cor} \label{bendingcurve} Let $k \geq 3$, let  $0 < \rho \leq \rho_0$ and $\sigma_0> 0$ be chosen as in \pref{main_est} and let $0 < \rho' \leq \rho/2$.
Let $(x,y) \in \R^2$ with $x > 0$  and $(u,v) \in S^1$ with $u, v > 0$ satisfying $|(x,y)| \leq \rho'$ and $v/x \geq\max \{ \sigma_0, \frac{\pi}{2 \lambda \rho'} \}$ where $\lambda = \frac{k-2}{4}$.
Then the curve $\gamma \colon [R,0] \to \R^2$ constructed in \lref{bc} is controlled and of extent $2\rho' \leq \rho$.
Moreover, $\scal_{F_{\gamma}} > 0$ on $\Sigma_{\gamma}$.
\end{cor}

\begin{proof}
The curve $\gamma$ is controlled by \lref{bc} \eqref{bb} and of extent $2 \rho'$ since $|\gamma(0)| \leq \rho'$, $|R| < \frac{\pi x}{2 \lambda v} \leq \rho'$ and $\gamma$ is of unit speed.
It follows from the proof of \lref{bc} that $\sigma$ is decreasing on $(R,0]$.
As $\sigma(0) = v/x \geq \sigma_0$, this implies $\sigma \geq \sigma_0$ on $[R,0]$ and hence $\scal_{F_{\gamma}} > 0$ on $\Sigma$ by \pref{main_est}.
\end{proof}

Finally, we are able to prove the main result of this section.
Roughly speaking, it says that we can choose scalar positive bending profiles which interpolate between the normally spherical immersions near $S$ resulting from \pref{local_deform} and scalar positive immersions which are ``parallel'' to the normal field $\xi$.
This is done by means of a suitable bending profile $\gamma$ as in Figure \ref{figure:bending} and is an essential ingredient for completing the scalar positive extrinsic surgery in \sref{extrsurg}.

\begin{prop}[Construction of bending profiles]\label{bending_curve}
Let $k \geq 3$.
There exists $0 < \rho \leq \rho_0$ such that for all $0 < \rho' \leq \rho/2$  there exists $\tau_0 > 0$ with the following property:
For all $\tau \geq \tau_0$ and all $0 < \rho'' \leq \min \{ \rho', \frac{\pi}{2\tau}\}$ there exists a regular smooth curve $\gamma = (a,b) \colon [R,0] \to \R^2$ of extent $2\rho'$ satisfying:
\begin{enumerate} [ (i)]
 \item \label{scalpos} The immersion $F_{\gamma} \colon \Sigma \to \R^N$ is scalar positive;
 \item \label{nearR} $\gamma(s) = \tau^{-1} \big( \sin(\tau (\rho'' + s ) ), 1 - \cos(\tau (\rho'' + s))\big)$ near $s=0$;
 \item \label{near0} $\gamma(s) =(a(R), b(R)+ s-R)$ near $s=R$.
 \end{enumerate}
\end{prop}
\begin{proof} We claim that the assertion holds for $\rho$ from \pref{main_est}.
Let $0 < \rho' \leq \rho/2$ and set $\tau_0 := \max \{\sigma_0 , \frac{\pi}{2 \lambda \rho'} \} $ with $\sigma_0$ from \pref{main_est}.
Pick $\tau \geq \tau_0$ and $0 < \rho'' \leq \min \{ \rho', \frac{\pi}{2\tau} \} $.

For $(x,y):=\tau^{-1}\big(\sin(\tau \rho''), 1-\cos(\tau \rho'')\big)$ and $(u,v):=(\cos(\tau \rho''), \sin(\tau \rho''))$ \cref{bendingcurve} applies since $0 < \tau \rho'' \leq \frac{\pi}{2}$, hence $u,v > 0$,  $|(x,y)| \leq \rho'' \leq \rho'$ and $v/x = \tau \geq \tau_0 \geq \max \{ \sigma_0 , \frac{\pi}{2 \lambda \rho'} \}$.
The resulting curve $\gamma \colon [R, 0] \to \R^2$ is regular, of extent $2\rho'$ and satisfies \eqref{scalpos}.
The proof will be completed once we deform $\gamma$ near $0$ in such a way that \eqref{nearR} holds as well, \eqref{near0} being treated in an analogous manner.

In order to do this, let $\varepsilon : = \min \{ |R/2| , \rho''/2\} >0$ and consider the continuous family $\gamma_t \colon [-\varepsilon,0] \to \R^2$, $t \in [0,1]$, of regular smooth curves of extent $2\rho'$ defined by
\[
\gamma_t(s) = (a_t(s), b_t(s)) := (1-t) \, \gamma (s) + t \, \tau^{-1} \big( \sin(\tau (\rho'' + s ) ), 1 - \cos(\tau (\rho'' + s))\big) \, .
\]
Then the $1$-jet $j^1 \gamma_t(0)$ is constant in $t$ by \lref{bc} \eqref{xx} and \eqref{aa}, and hence the same holds for $\sigma_t(0) := \frac{b_t'(0)}{a_t(0)}$.

For the curvature $\kappa_t (0)$ of $\gamma_t$ at $s = 0$ we obtain the linear interpolation
$$
\kappa_t(0) = (1-t) \, \kappa(0) + t \, \tau = - (1-t) \, \lambda \sigma_t(0) + t \, \sigma_t(0) \, .
$$
In particular, $- \lambda \sigma_t(0) \leq \kappa_t(0) \leq \sigma_t(0)$ for all $t \in [0,1]$, and thus each $\gamma_t$ is controlled and of unit speed at $s = 0$.

\pref{main_est} shows that $\scal_{F_{\gamma_t}} > 0$ along $S_1(\NUMS) \times \{0\} \subset \Sigma$ for  all $t \in [0,1]$ by the choice of $\sigma_0$.
Passing to a smaller $\varepsilon >0$ if necessary, this implies that $F_{\gamma_t} \colon S_1(\NUMS) \times ( - \varepsilon, 0] \to \R^{N}$ is a scalar positive immersion for all $t \in [0,1]$.
Since $\gamma$ being regular and of extent $2\rho'$ and $F_{\gamma}$ being scalar positive defines an open partial differential relation on the $2$-jets $j^2 \gamma$ of smooth curves $[R,0] \to \R^2$ and since $j^1 \gamma_t(0)$ is constant in $t$, the local flexibility lemma \cite{bh}*{Theorem 1} applies.
Hence there exists $0 < \varepsilon_0 < \varepsilon$ and a continuous family  $\Gamma_t \colon [R,0] \to \R^2$, $t \in [0,1]$, of regular smooth curves of extent $2\rho'$ with $\Gamma_0 = \gamma$ and such that the $\Gamma_t$ coincide with $\gamma_t$ on $(- \varepsilon_0,0]$,  are constant in $t$ on $[R,- \varepsilon]$ and induce scalar positive immersions $f_{\Gamma_t }\colon \Sigma \to\R^{N}$.
We now replace $f$ by $\Gamma_1$, thus achieving \eqref{nearR}.
\end{proof}

\usetikzlibrary{decorations.markings}
\tikzset{middlearrow/.style={
 decoration={markings,
 mark= at position 0.6 with {\arrow{#1}} ,
 },
 postaction={decorate}
 }
}
\begin{figure}
\begin{tikzpicture}[scale=1.5]
\filldraw[gray!4!white] (0,0) circle[radius=1.8];
\draw[gray!80!white, -> ] (210:0.6) -- (210:0.2);
\draw[gray!80!white, -> ] (210:1.2) -- (210:1.6);
\draw[gray!80!white] node at (210:0.9) {$2\rho'$};
\draw[->] (0,0)-- (1.5,0) -- node[near end, above] {$a$} (2,0);
\draw[->] (0,-2) -- (0,1) -- node[near end, right] {$b$} (0,2);
\begin{scope}[shift={(0,2)}]
\draw[dotted, thick, black] (270:2) arc [radius = 2cm, start angle = 270 , end angle = 315 ];
\draw[very thick, blue] ( 315:2 ) arc [radius = 2cm, start angle = 315 , end angle = 340 ];
\draw[very thick,dotted, blue] ( 340:2 ) arc [radius = 2cm, start angle = 340 , end angle = 350 ];
\draw[red, very thick, middlearrow={latex reversed} ] (315:2) to [out=225, in=90] (0.5,-3) ;
\filldraw[red] (315:2) circle[radius=1pt];
\end{scope}
\begin{scope}[shift={(0.5,2)}];
\draw[red] node at (315:2) {$\gamma(0)$};
\end{scope};
\draw[very thick, green] (0.5,-1) to (0.5,-1.5) ;
\draw[very thick, green, dotted] (0.5,-1.5) to (0.5,-1.9);
\filldraw[red] (0.5,-1) circle[radius=1pt];
\filldraw[black] (0,0) circle[radius=1pt];
\draw[red] node at (1,-1) {$\gamma(R)$};
\draw[red] node at (0.5,-0.3) {$\gamma$};
\end{tikzpicture}
\caption{In red: the bending profile in \pref{bending_curve}.}
\label{figure:bending}
\end{figure}

\begin{rem} \label{compareGL} Our discussion may be adapted to provide an  alternative approach to the surgery lemma in \cite{gl} by considering the embedding $\tilde F_{\gamma}\colon \Sigma \hookrightarrow M \times \R$,
\[
 (q, \omega, s) \mapsto \big( \exp^{\perp} (a(s) \, \omega), b(s) \big) \in M \times \R ,
\]
which is defined whenever the extent of $\gamma$ is smaller than the normal injectivity radius of $S \subset M$.
Roughly speaking, in \eqref{defFg} the normal field $\xi$  is replaced by the unit vector field $\partial_t \in \Gamma(T( M \times \R))$ pointing in the $\R$-direction.
In this situation the generalized Gauss lemma for $\exp^{\perp}$ implies that  $N(p)=N^{\perp}(p)$ for all $p \in \Sigma$ with respect to the embedding $\tilde F_{\gamma}$, rendering an estimate as in  Lemma \ref{estnormal} obsolete.

Note that contrary to \cite{gl}*{Equation (1') on p.~429}, our  \pref{main_est} does not yield a positive lower bound for $\scal_{F_\gamma}$ in case $b'=0$ (hence $\sigma = 0$).
This is related to the fact that  the target of $F_\gamma$ is flat $\R^N$, whereas the one of $\tilde F_{\gamma}$ is scalar positive.
Hence  in our extrinsic setting the ``initial stage'' of the bending process requires a different approach  than in \cite{gl}.
This is provided by our \pref{local_deform} which relies on the local flexibility lemma \cite{bh}*{Theorem 1}.
\end{rem}

\section{Extrinsic scalar positive surgery} \label{extrsurg}

Here we combine the previous constructions in order to perform the extrinsic surgery.
At the end of this section we give the proofs of our two main results in the introduction.

Let $f \colon M \to \R^{N}$ be a scalar positive immersion where $M$ is of dimension $n$ and let $S \subset M$ be a closed embedded submanifold of dimension $d$ with normal bundle $\nu_S^M \to S$.
Assume that $S \subset M$ has codimension   $n - d = k \geq 3$.
The following brings together the main results of Sections \ref{nomsph} and \ref{bending}.

\begin{prop} \label{ultimate}
For all $\varepsilon, \lambda_0 > 0$ there exist constants $\rho, \tau >0$ with the following properties:
\begin{enumerate}[(i)]
 \item The normal exponential map $\exp^{\perp} \colon \NUMS \to M$ induces a diffeomorphism $D_{\rho}(\nu_S^M) \approx \overline{U_{\rho}(S)}$;
 \item \label{magic_deform} There is a continuous family $f_t \colon M \to \R^N$,  $t \in [0,1]$, of scalar positive immersions   such that  $f_0=f$ and such that  for all $(q, \omega, s) \in U_{\rho}(S)$  we have
 \[
 f_1(q, \omega, s) = f(q) + \tau^{-1} \sin ( \tau s) \omega + \tau^{-1}( 1- \cos( \tau s))\xi(q) \, ;
 \]
 \item \label{bendfinal} There exists $R < 0$ and a regular smooth curve $\gamma = (a,b) \colon [R,0] \to \R^2$ satisfying
$$
\gamma(s) =
\begin{cases}
 \tau^{-1} \big( \sin (\tau ( \rho +s)), 1 - \cos ( \tau ( \rho +s )) \big) & {\rm near \ s=0},\\
 \left(a(R), b(R) + s-R\right) & {\rm near \ } s=R,
 \end{cases}
$$
where $0 < a(R) < \lambda_0$ and $- \varepsilon < b(R) < \varepsilon$, and such that the map $F_{\gamma} \colon S_{1}(\NUMS) \times [R,0] \to \R^N$, $F_\gamma(q, \omega,s)= f(q) + a(s) \omega + b(s) \xi(q)$, is a scalar positive immersion.
\end{enumerate}
\end{prop}

\begin{proof} Choose $ \rho$ as in \pref{bending_curve} and set $\rho' := \min \{ \rho/2, \lambda_0/2, \varepsilon/ 2\}$.
For this $\rho'$ let $\tau_0$ be chosen as in \pref{bending_curve}.
By \pref{local_deform} there exist $\tau \geq \tau_0$ and  $0 < \rho'' \leq \min\{ \rho', \frac{\pi}{2\tau} \}$ such that $f$ can be deformed into $f_1$ through scalar positive immersions in such a way that the formula for $f_1$ in \eqref{magic_deform} holds for all $(q, \omega, s) \in U_{\rho''}(S)$.
Furthermore, by \pref{bending_curve}, we find $\gamma$ of extent $2 \rho' \leq \min\{ \lambda_0, \varepsilon\}$ with properties as described in \eqref{bendfinal}, except that  the formula for $\gamma(s)$ holds with $\rho$ replaced by $\rho''$.
We conclude that all the assertions of \pref{ultimate} hold for $\rho:= \rho''$ and $\tau$.
\end{proof}

Now let $S$ be additionally diffeomorphic to the unit $d$-sphere $\mathbb{S}^d \subset \R^{d+1}$ and fix a diffeomorphism $S \approx \mathbb{S}^d$.
Furthermore, let the normal bundle $\nu_S^M \to S$ be trivialisable and fix an orthonormal frame $(e_1, \ldots, e_k)$ of $\nu_S^M$.
Finally, let $F\colon \mathbb{D}_{1+\varepsilon}^{d+1} \to \R^{N}$ be an immersion of the closed $(1+\varepsilon)$-disc in $\R^{d+1}$ for some $0 < \varepsilon < 1$ together with a linear independent family of sections $(E_1, \ldots, E_k)$ of the trivial bundle $\mathbb{D}_{1+\varepsilon}^{d+1} \times \R^{N} \to \mathbb{D}_{1+\varepsilon}^{d+1}$  which spans a bundle having zero intersection with $T\mathbb{D}_{1+\varepsilon}^{d+1}$ and is compatible with $f$ and $(e_1, \ldots, e_k)$ in the following sense: For all $\omega \in S \approx \mathbb{S}^d = S_1(\R^{d+1} ) \subset \mathbb{D}_{1+\varepsilon}^{d+1}$ and $r \in [1 - \varepsilon, 1+ \varepsilon ]$, we have
\[
 F(r\omega) = f(\omega) + (r-1) \xi(\omega) \, , \quad E_i (r\omega) = e _i(\omega) \text{ for } i = 1, \ldots, k \, .
\]
Note that under these conditions  the family $(E_1, \ldots, E_k)$ is in general not normal to $T\mathbb{D}_{1+\varepsilon}^{d+1}$.

By an argument similar to the proof of \pref{tub_imm} we find $\lambda_0 > 0 $ such that for all $0 < \lambda \leq \lambda_0$ the map $\mathscr{F}_{\lambda} \colon\mathbb{D}_{1+\varepsilon}^{d+1} \times \skmu \to \R^N$,
\[
 \mathscr{F}_{\lambda} \big(q, v_1, \ldots, v_{k} \big) := F(q) + \lambda \, \sum_{i=1}^{k} v_i E_i(q),
\]
is a scalar positive immersion.

We apply  \pref{ultimate} and consider the smooth manifold
\begin{equation} \label{piecetogether}
\hat M := M \setminus U_\rho(S) \; \bigcup \; \Sigma_{\gamma} \; \bigcup \; \mathbb{D}^{d+1}_{1+b(R) } \times \skmu
\end{equation}
where we glue
$$
\partial \big( M \setminus U_{\rho}(S) \big)  \approx S_{\rho}(\NUMS)
\ \ \ \rightleftharpoons\ \ \
S_1(\NUMS) \times \{ 0 \} \subset \partial \Sigma_\gamma$$
along the  dilation map $S_{\rho} (\NUMS) \approx S_1(\NUMS)$ and
$$
S_1(\NUMS) \times \{R\} \subset \partial \Sigma_\gamma
\ \ \ \rightleftharpoons\ \ \
\partial \big( \mathbb{D}^{d+1}_{1+b(R) } \times \skmu \big)
$$
along the map
\[
S_1(\NUMS) \times \{R\} \approx \mathbb{S}^d \times \skmu \approx \partial \big( \mathbb{D}^{d+1}_{1+b(R)} \times \skmu \big)
\]
which is induced by the given diffeomorphism $S \approx \mathbb{S}^d$,  the dilation map $\mathbb{S}^d = S_1(\R^{d+1}) \approx S_{1+b(R)} (\R^{d+1}) = \partial \mathbb{D}^{d+1}_{1+b(R)}$ and the frame $(e_1, \ldots, e_k)$.

As usual we say that the manifold $\hat M$ is obtained from $M$ by a {\em surgery} along $S \subset M$ with respect to the normal frame $(e_1, \ldots, e_k)$.
 By \pref{ultimate} \eqref{magic_deform} and \eqref{bendfinal}, the maps $f_1$ on $M \setminus U_{\rho}(S)$, $F_{\gamma}$ on $\Sigma_{\gamma}$ and $\mathscr{F}_{a(R)}$ on $\mathbb{D}^{d+1}_{1+b(R)} \times \skmu $ are compatible at the gluing regions in $\hat M$ and combine to a scalar positive smooth immersion $\hat f \colon \hat M \to \R^N$.
In terms of bending profiles near $S \subset M$ the images of the first, second and third pieces in \eqref{piecetogether} under $\hat{f}$ correspond to the blue, red and green pieces in Figure \ref{figure:bending}.

\medskip

We finish by proving our main results.

\begin{proof}[Proof of \tref{extr_surg}]
Since $2d +1 \leq N$ the immersion $\mathbb{S}^d \approx S \stackrel{f|_{S}}{\rightarrow}  \R^N$ extends to an immersion $F \colon \mathbb{D}_{1+\varepsilon}^{d+1} \to \R^N$ by \cite{s2}*{Theorem B} for some $0 < \varepsilon < 1$ such that $F(r\omega) = f(\omega) + (r-1) \xi(\omega)$ for all $r \in [1- \varepsilon, 1+ \varepsilon]$ and $\omega \in S \approx \mathbb{S}^d$.

The manifold $\mathbb{D}^{d+1}_{1+ \varepsilon}$ is contractible and hence the normal bundle $\nu_F \to \mathbb{D}^{d+1}_{1+ \varepsilon}$, which is of rank $N - d-1$, is trivial.
Since the Stiefel manifold $V_{n-d}(\R^{N-d-1})$ of $(n-d)$-frames in $\R^{N-d-1}$ is $(N-n-2)$-connected and $d \leq N-n-2$, the family $(e_1 , \ldots, e_k)$ (recall $k = n - d$) extends (after possibly decreasing $\varepsilon$) to a linear independent family of sections $(E_1, \ldots, E_k)$ of the trivial $\R^N$-bundle over $\mathbb{D}^{d+1}_{1+\varepsilon}$ with properties as described after the proof of Proposition \ref{ultimate}.
Now the extrinsic surgery construction may be carried out to obtain the required scalar positive immersion $\hat f \colon \hat M \to \R^N$.

The last assertion follows since $\rho_0$ in \pref{local_deform} can be chosen arbitrarily small.
\end{proof}

\begin{proof}[Proof of \tref{main}]
Assume that $M$ is spin.
Since $\alpha(M)=0$, by \cite{ks}*{Proposition 3.3} $M$ is spin bordant to the total space of a fibre bundle $\HP^2 \hookrightarrow V \to B $ with structure group $\Sp(3)$ over a closed spin manifold $B$.
There exists a scalar positive immersion $V \to \R^{2n-1 +\delta(n)}$ as described in Example \ref{Veronese}.
Since $M$ and $V$ are spin bordant, $M$ is simply connected and $\dim M \geq 5$, we can obtain $M$ from $V$ by a finite number of surgeries in codimensions at least $3$, using Smale's handle cancellation technique, compare \cite{gl}.
\tref{main} now follows from \tref{extr_surg} as $2n - 1 + \delta(n) \geq 2n - 1 = n + (n-d) + d - 1 \geq n + 3 + d - 1 = n+d +2$.

If $M$ is not spin, F\"uhring \cite{fu}*{Theorem 1.1} used the methods of \cite{st} and \cite{ks} to show that $M$ is oriented bordant to the total space of a fibre bundle $\CP^2 \hookrightarrow V \to B$ with structure group  $\UU(3) \rtimes \ZZ/2$  over a closed oriented manifold $B$.
There exists a scalar positive immersion $V \to \R^{2n-1 +\delta(n)}$ as described in Example \ref{Veronese}.
Since $M$ and $V$ are oriented bordant,  $M$ is simply connected and not spin and $\dim M \geq 5$, we can obtain $M$ from $V$ by a finite number of surgeries in codimensions at least $3$.
Hence \tref{main} again follows from \tref{extr_surg}.
\end{proof}

\end{document}